\documentclass[11pt]{amsart}

\usepackage{verbatim}
\usepackage{graphicx}
\usepackage{graphics}
\usepackage{latexsym}
\usepackage{amssymb,amsmath}
\usepackage{amsthm}
\usepackage{amsfonts}
\usepackage{multicol}
\usepackage{bussproofs}
\usepackage{enumerate}
\usepackage{xcolor}
\usepackage[latin1]{inputenc}
\usepackage{pictex}
\usepackage{url}

\usepackage{ragged2e} 

\begin{document}



\numberwithin{equation}{section}

\newtheorem{definicion}{Definition}[section]
\newtheorem{Definition}[definicion]{Definition}
\newtheorem{definition}[definicion]{Definition}
\newtheorem{Lemma}[definicion]{Lemma}
\newtheorem{Notation}[definicion]{Notation}
\newtheorem{Theorem}[definicion]{Theorem}
\newtheorem{theorem}[definicion]{Theorem}
\newtheorem{Corollary}[definicion]{Corollary}
\newtheorem{corollary}[definicion]{Corollary}
\newtheorem{lema}[definicion]{Lemma}
\newtheorem{notacion}[definicion]{Notation}
\newtheorem{teorema}[definicion]{Theorem}
\newtheorem{corolario}[definicion]{Corollary}
\newtheorem{ejemplo}[definicion]{Example}
\newtheorem{observacion}[definicion]{Observation}
\newtheorem{proposicion}[definicion]{Proposition}
\newtheorem{Remark}[definicion]{Remark}
\newenvironment{Proof}{\noindent\bf Proof \rm}{$\hfill
\square$}

\newcommand{\debaj}[2]{ #1 \to_H #2}

 \newcounter{thlistctr}
 \newenvironment{thlist}{\
 \begin{list}%
 {\alph{thlistctr}}%
 {\setlength{\labelwidth}{2ex}%
 \setlength{\labelsep}{1ex}%
 \setlength{\leftmargin}{6ex}%
 \renewcommand{\makelabel}[1]{\makebox[\labelwidth][r]{\rm (##1)}}%
 \usecounter{thlistctr}}}%
 {\end{list}}

\title[Unorthodox algebras and their logics]{Unorthodox algebras and their associated unorthodox logics}

\author{Hanamantagouda P. Sankappanavar}

\date{}

\address{\llap{**\,}Department of Mathematics\\
                              State University of New York\\
                              New Paltz, New York, 12561\\
                              U.S.A.}

\email{sankapph@hawkmail.newpaltz.edu}


\date{}
\begin{abstract}
This paper grew out of our investigation into a simple, but natural, question:   
Can ``$F \to T$'' be different from ``F'' and ``T''?
To this end, we introduce five ``unorthodox'' algebras 
 that will play a major role, not only in providing 
a positive answer to the question, but also in their similarity to the 2-element Boolean algebra $\mathbf{2}$.  Yet, they are remarkably different from $\mathbf{2}$ in many respects.
In this paper, we will examine these five algebras both algebraically and logically.   
We define, and initiate an investigation 
into, a subvariety, called $\mathbb{RUNO}1$, of the variety of De Morgan semi-Heyting algebras and show that $\mathbb{RUNO}1$ is, in fact, the variety generated by the five algebras.  
Then several applications of this theorem are given. It is shown that $\mathbb{RUNO}1$ is a discriminator variety and that all five algebras are primal.  It is also shown that 
every subvariety of $\mathbb{RUNO}1$ satisfies the Strong Amalgamation Property and the property that epimorphisms are surjective (ES).
  It is shown that the lattice of subvarieties of $\mathbb{RUNO}1$ is a Boolean lattice of 32 elements.  The bases for all the subvarieties of $\mathbb{RUNO}1$ are also given.  
 We introduce a new logic called $\mathcal{RUNO}1$ and show that it is algebraizable with the variety $\mathbb{RUNO}1$ as its equivalent algebraic semantics.
  We then present axiomatizations for all 32 axiomatic extensions of 
$\mathcal{RUNO}1$ and deduce that all the axiomatic extensions are decidable.
The paper ends with some open problems.  
\end{abstract}

\maketitle

\thispagestyle{empty}

\section{Introduction}

 Let ``true'' and ``false'' abbreviate to ``T`` and ``F''  (or, ``1'' and ``0'') respectively, and let $\to$ denote an implication. 
Let us start with a very simple, but natural, question:   
 
{\bf Question A:}  What is ``$F \to T''$?, or equivalently, what is ``$0 \to 1''?$
 
 Of course, a student who is taking the first course in logic would, in most likelihood, be told by her/his professor that $F \to T =T,$ or in the other notation, $0 \to 1= 1$, as the answer to the question. This answer, together with $1 \to 1=1$, $0 \to 0=1$, and $1 \to 0=0$, would lead, of course, to the well-known ``orthodox'' (i.e. Boolean) implication on the Boolean algebra $\mathbf{2}$ and, thereby, to the the variety $\mathbb{V}(\mathbf{2})$ of Boolean algebras and the associated ``orthodox'' (i.e. classical) logic.  
 A curious student, however, might ask his/her professor the following follow-up question, in response to the professor's previous answer:
 
{\bf Question B: Can ``$0 \to 1$'' be different from 1? }

If the student was taking the logic course before the middle of the nineteenth century, it is very likely that he/she would have been told by the professor that it couldn't be different from T (or 1); but 
 if the student is taking the course in 2025, the professor would have certainly answered question B by saying, ``YES, indeed, it is possible''.  This paper grew out of an attempt to give an elaborate and convincing justification to the answer ``YES.'' 

In order to answer Question B completely satisfactorily, however, we need to first refine the Question B into two questions as follows:

{\bf Question B1:}   Can ``$0 \to 1 =0$, so that the resulting logic is ``interesting''? 

 {\bf Question B2:}   Can $0 \to 1$ be different from both $0$ and $1$, so that the resulting logic is also ``interesting'' )? 

  Let us first consider, briefly, an answer to {\bf Question B1}.  For this purpose, consider the following 2-element   
  algebra, with $0 < 1$, in the language $\langle \lor, \land, \to, 0, 1\rangle$:\\

$\mathbf{\overline{2}}$ :                        
\begin{tabular}{r|rr}
	$\to$ & 0 & 1\\
	\hline
	0 & 1 & {\color{red}0} \\
	1 & 0 & 1
\end{tabular} \hspace{.4cm}  \\

\medskip
 Observe that it is indeed the case that $0 \to 1=0$ in the algebra $\mathbf{\overline{2}}$.  So, we could call this algebra ``{\bf anti-Boolean}.''   YET,  $\mathbf{\overline{2}}$, and the variety $\mathbb{V}(\mathbf{\overline{2}})$ that it generates, as well as the corresponding logic are quite interesting in the sense that they share several important properties with the Boolean (i.e., orthodox) algebra $\mathbf{2}$, the variety $\mathbb{V}(\mathbf{2})$ and its corresponding logic namely classical logic.  For instance,  
\begin{itemize}
\item $\mathbf{\overline{2}}$ is a primal algebra,  just as the Boolean algebra $\mathbf{2}$ is.

\item $\mathbb{V}(\mathbf{\overline{2}})$ has an interesting logic.  It is a {\bf connexive logic} in the sense that Aristotle's Theses: $\vdash \neg(\neg\alpha \to \alpha)$, and $\vdash \neg(\alpha \to \neg\alpha)$ and Boethius Theses: $\vdash (\alpha \to  \beta) \to \neg (\alpha \to \neg \beta)$ and  $(\alpha \to \neg\beta) \to \neg (\alpha \to \beta)$ are theorems in that logic.   

\item More interestingly, the variety $\mathbb{V}(\mathbf{\overline{2}})$ turns out to be  
{\bf term-equivalent} to $\mathbb{V}(\mathbf{2})$, the variety of Boolean algebras!  Hence, 
we could say that the classical logic is ``connexive'' too.
 \end{itemize}
 Thus the answer to {\bf Question B1} is a definite ``YES''.  
 For all these claims, and much more, about $\mathbf{\overline{2}}$, $\mathbb{V}(\mathbf{\overline{2}})$, its associated connexive logic, and other related connexive logics, we refer the reader to the recent papers \cite{CoSa25} and \cite{FaLePa24}.

  The main focus of this paper\footnote{The contents of this paper were presented by the author in invited lectures at the conference on Non-Classical Logics, Lodz, Poland, Sept 5-8, 2024 and at the International Conference on Algebra, Logic, and their Applications, Yerevan, Armenia, 13-19 October 2024, as well as at Mathematical Congress of Americas, Miami, Florida, 21-25 July 2025.} is to address {\bf Question B2}.  Since the variety of Boolean algebras is known to have an axiomatization in the language {\bf L} = $\langle \lor, \land, \to, ', 0, 1\rangle$, we will address {\bf Question B2} in the language {\bf L}, together with some of its far reaching ramifications.
The following FIVE algebras in the language {\bf L} will play a major role in providing 
a positive answer to {\bf Question B2}.  These algebras
  can be viewed as logical matrices with 1 as  the only designated truth value.  If we think of the classical logic as ``orthodox'', then we could, informally speaking, think of the following algebras as ``unorthodox'' as their $\to$ operation is such that $0 \to 1$ is a fixed point under the (unary) De Morgan operation $'$ and hence is
  neither $1$ nor $0$.  (The precise definition of `unorthodox'' will be given later in the paper.)   The first four of these five algebras are chain-based and hence $\land$, $\lor$ are not shown.

 \subsection{``Unorthodox Algebras''}
 
 \
 
 \medskip
 All five algebras are of type {\bf L}= $\langle \lor, \land, \to, ', 0, 1\rangle$. The first four have the 3-element chain as their lattice-reduct and the fifth one has the 4-element Boolean lattice as its lattice-reduct.
 \medskip
 \begin{itemize}
\item
 {\bf ALGEBRA} $\mathbf{A1}$: \quad $0<2<1,$ \\
 \vspace{.24cm}
 
\begin{tabular}{r|rrr}
$\to$ & 0 & 1 & 2\\
\hline
    0 & 1 & {\color{red}2} & 1 \\
    1 & 0 & 1 & 2 \\
    2 & 0 & 1 & 1
\end{tabular} \hspace{1.7cm}
\begin{tabular}{r|rrr}
$'$ & 0 & 1 & 2\\
\hline
   & 1 & 0 & 2
\end{tabular} \hspace{1.6cm}
\vspace{.31cm}
\item
\vspace{1.7cm}
{\bf ALGEBRA} $\mathbf{A2}$:  \quad $0<2<1,$\\
\vspace{.24cm}

\begin{tabular}{r|rrr}
$\to$: & 0 & 1 & 2\\
\hline
    0 & 1 & {\color{red}2} & 1 \\
    1 & 0 & 1 & 2 \\
    2 & 0 & {\color{red}2} & 1
\end{tabular} \hspace{1.7cm}
\begin{tabular}{r|rrr}
$'$: & 0 & 1 & 2\\
\hline
   & 1 & 0 & 2
\end{tabular} \hspace{.5cm}\\
\medskip
\end{itemize}

\begin{itemize}
\newpage
\item
 {\bf ALGEBRA} $\mathbf{A3}$:   \quad $0<2<1,$\\
 \vspace{.23cm}

\begin{tabular}{r|rrr}
$\to$: & 0 & 1 & 2\\
\hline
    0 & 1 & {\color{red}2} & {\color{red}2}\\
    1 & 0 & 1 & 2 \\
    2 & 0 & 1 & 1
\end{tabular} \hspace{1.7cm}
\begin{tabular}{r|rrr}
$'$: & 0 & 1 & 2\\
\hline
   & 1 & 0 & 2
\end{tabular} \hspace{.5cm}
\vspace{.2cm}
\item
 {\bf ALGEBRA} $\mathbf{A4}$:  \quad $0<2<1,$\\

\vspace{.23cm}

\begin{tabular}{r|rrr}
$\to$ & 0 & 1 & 2\\
\hline
    0 & 1 & {\color{red}2} & {\color{red}2} \\
    1 & 0 & 1 & 2 \\
    2 & 0 & {\color{red}2} & 1
\end{tabular} \hspace{1.7cm}
\begin{tabular}{r|rrr}
$'$ & 0 & 1 & 2\\
\hline
   & 1 & 0 & 2
\end{tabular} \hspace{.7cm}
\vspace{.6cm}
\end{itemize}
\begin{itemize}
\item
{\bf ALGEBRA} $\mathbf{A5}$:  \quad  $2 \lor 3=1$, $2 \land 3=0$,\\
\vspace{.23cm}

\begin{tabular}{r|rrrr}
$\to$ & 0 & 1 & 2 & 3\\
\hline
    0 & 1 & {\color{red}2} & 1 & {\color{red}2} \\
    1 & 0 & 1 & 2 & 3 \\
    2 & 3 & {\color{red}2} & 1 & {\color{red}0} \\
    3 & 2 & 1 & 2 & 1
\end{tabular} \hspace{1cm}
\begin{tabular}{r|rrrr}
$'$ & 0 & 1 & 2 & 3\\
\hline
   & 1 & 0 & 2 & 3
\end{tabular} \hspace{.5cm}
\end{itemize}

\begin{Remark}  
\begin{thlist}
\item[a] All these algebras share the following properties with the the Boolean (i.e., orthodox) algebra $\mathbf{2}$:
\begin{thlist}
  \item[1] $0 \to 1=  2$ (different from $0$ and $1$).
 
 \item[2] $(0 \to 1)' =  0 \to 1.$  Hence, we can consider these algebras as ``unorthodox.'' 
 
 \item[3] They have no proper subalgebras.  In particular, none of the five algebras have the Boolean algebra $\mathbf 2$ as a subalgebra.  
 
 \item[4] They have no nontrivial automorphism.
 \end{thlist} 
\item[b] The elements in red on the $\to$ table of the first four algebras indicate the places where they differ from the $\to$ table of the 3-element Heyting algebra, while those of the fifth algebra indicate where they differ from the $\to$ table of the 4-element Boolean algebra.
\end{thlist} 

\end{Remark}

In this paper, we will examine these five algebras both algebraically and logically.   So, we first investigate the variety generated by the above five algebras; more specifically, we will examine the following problem:

 {\bf PROBLEM C:  Find an equational base for the variety generated by the above
 five algebras.}

We then study the structure of the lattice of subvarieties of the variety $\mathbb{V}(\mathbf{A1}, \mathbf{A2}, \mathbf{A3}, \mathbf{A4}, \mathbf{A5})$ generated by the five algebras.  We also introduce logics associated with the subvarieties of this variety
and investigate the properties, such as amalgamation property, being a discriminator variety, decidability of the logics and so on.

More specifically, the paper is organized as follows:  Section 2 recalls definitions, notation and preliminary results about De Morgan semi-Heyting algebras that are needed later in the paper.  In Section 3, we define and initiate an investigation 
into a subvariety, called $\mathbb{RUNO}1$, of the variety of De Morgan semi-Heyting algebras.  Section 4 presents a characterization of the 
subdirectly irreducible algebras in $\mathbb{RUNO}1$ and as a consequence, we show that $\mathbb{RUNO}1$ is, in fact, the variety generated by the five algebras, thus solving PROBLEM C. Then several applications of this theorem are given. It is shown that $\mathbb{RUNO}1$ is a discriminator variety and that all five algebras are primal.  It is also shown that 
every subvariety of $\mathbb{RUNO}1$ satisfies the Strong Amalgamation Property and the property that epimorphisms are surjective (ES).
 Sections 5 and 6 give further applications.  It is shown, in Section 5, that the lattice of subvarieties of $\mathbb{RUNO}1$ is a Boolean lattice of 32 elements.  We also present bases for all 32 subvarieties of $\mathbb{RUNO}1$.   
 In section 6 we give several more identities that hold in $\mathbb{RUNO}1$ and in its subvarieties. 
 
The rest of the paper deals with the variety $\mathbb{RUNO}1$ from a logical point of view.   Section 7 will recall, by way of a quick introduction to implicative logics in the sense of Rasiowa, and to abstract algebraic logic due to Blok and Pigozzi, the notions of algebraizability of a logic and its equivalent algebraic semantics as well as the most fundamental results connecting them.  In Section 8, we recall from \cite{CoSa22a} a known algebraizable logic called $\mathcal{DMSH}$ corresponding to  De Morgan semi-Heyting algebras, which plays a crucial role in the rest of the paper.  We then introduce a new logic called $\mathcal{RUNO}1$ as an axiomatic extension of the logic $\mathcal{DMSH}$.
It is then deduced that the logic $\mathcal{RUNO}1$ is implicative and hence algebraizable, 
with the variety $\mathbb{RUNO}1$ as its equivalent algebraic semantics. Section 9 presents bases for all axiomatic extensions of 
$\mathcal{RUNO}1$, from which it follows that all its axiomatic extensions are decidable.
The paper ends with some open problems in Section 10.

\section{Preliminaries}

In this section, we recall some useful results regarding De Morgan semi-Heyting algebras. 
Observe that all the five algebras mentioned above have semi-Heyting reducts.

\indent Semi-Heyting algebras were introduced by the author in 1984-85 as an abstraction of Heyting algebras.  The first results on 
these algebras, however, were not published until 2007 (see \cite{Sa07}).

An algebra $L = \langle L, \land, \lor, \to, 0, 1\rangle$ is a semi-Heyting algebra if the following conditions hold:
\begin{enumerate}[(SH1)]
	\item $\langle L, \land, \lor, 0,1\rangle$ is a lattice with $0$ and $1$,
	\item $x \land (x \to y) \approx x \land y$,
	\item $x \land (y \to z) \approx x \land [(x \land y) \to (x \land z)]$,
	\item $x \to x \approx 1$.
\end{enumerate}
A semi-Heyting algebra is a Heyting algebra if it satisfies:
$$(x \land y) \to x \approx 1.$$
We will denote the variety of semi-Heyting algebras by $\mathbb{SH}$ and that of Heyting algebras by $\mathbb{H}.$

Semi-Heyting algebras share some important properties with Heyting algebras; for instance, semi-Heyting algebras are distributive and pseudocomplemented; the congruences on them are determined by filters and the variety of semi-Heyting algebras is arithmetical 
(for more such properties, see \cite{ACDV11, Ab11a, AbCoVa13, Sa07}). 

A key feature of semi-Heyting algebras is the following \cite{AbCoVa13}: Every semi-Heyting algebra $\langle A,\lor,\land,\to, 0, 1 \rangle$ gives rise naturally to a Heyting algebra $\langle A,\lor,\land,\to_H, 0, 1 \rangle$ by defining the implication $x\to_H y$ as $x\to(x\land y)$.\\

	The following lemmas will be useful in the sequel.
	
\begin{Lemma}	 Let  $\mathbf{L} = \langle L,  \land, \lor, \to, 0, 1\rangle$ be a semi-Heyting algebra.
For $a,b,c \in L,  a \land b \leq c$ if and only if  $a \leq b \to (b \land c)$.
\end{Lemma}

\begin{Lemma}{\rm \cite[Corollary 3.9]{Co11} }
Let  $\mathbf{L} = \langle L,  \land, \lor, \to, 0, 1\rangle$ be a semi-Heyting algebra.
For $a,b \in L$,  $a \to_H b = 1$ if and only if $a \leq b$. 
\end{Lemma}

\vspace{1cm}

\section{The variety $\mathbb{RUNO}1$}

In this section we define, and investigate, an important subvariety, called ``$\mathbb{RUNO}1$'' of the variety of dually hemimorphic semi-Heyting algebras.

\begin{Definition}
An algebra $\mathbf A=\langle A, \wedge, \lor, \to, ', 0,1\rangle$ is a semi-Heyting algebra with a dual hemimorphism (or dually hemimorphic semi-Heyting algebra) 	if $\mathbf A$ satisfies the following conditions:	
	\begin{enumerate}[$(E1)$]
		\item  $\langle A,\wedge, \lor, \to, 0,1\rangle $ is a semi-Heyting algebra, 
		\item $0' \approx 1$, \label{conditionSD2d1}
		\item $1' \approx 0$, \label{conditionSD2d2} 
		\item $(x \land y)' \approx x' \lor y'$  {\rm ($\land$-De Morgan law)}.  \label{infDMlaw} \label{conditionSD3d}
	\end{enumerate}
We will denote the variety of dually hemimorphic semi-Heyting algebras by $\mathbb{DHMSH}$.
An algebra $\mathbf A \in \mathbb{DHMSH}$ is a dually quasi-De Morgan semi-Heyting algebra if  $\mathbf A$ satisfies the following:

{\rm(DQD1)} $(x \lor y)'' \approx x'' \lor y''$;

{\rm(DQD2)} $x'' \leq x$.\\
The variety of dually quasi-De Morgan semi-Heyting algebras is denoted by $\mathbb{DQDSH}$.\\ 
An algebra $\mathbf A \in \mathbb{DQDSH}$ is said to be De Morgan if $\mathbf A \models x'' \approx x$.
An algebra $\mathbf A \in \mathbb{DHMSH}$ is said to be {\bf unorthodox} if 
$\mathbf A$ satisfies the axiom:  $(0 \to 1)' = 0 \to 1.$ \\
The variety of unorthodox De Morgan semi-Heyting algebras is denoted by $\mathbb{UNO}$.  
An algebra $\mathbf A \in \mathbb{UNO}$ is called regular if  $\mathbf A$ satisfies the identity:

$x \land x'{^*}' \leq y \lor y^*$ \quad (R).\\
The variety of regular unos is denoted by $\mathbb{RUNO}$. 
An algebra $\mathbf A \in \mathbb{DHMSH}$ is of level 1 if $\mathbf A$ satisfies the identity: 
$x \land x'^* \approx (x \land x'{^*})'^*$.\\
If  $\mathbb{V}$ is a subvariety of $\mathbb{DHMSH}$, then  $\mathbb{V}1$ denotes the subvariety of $\mathbb{V}$ consisting of algebras from $\mathbb{V}$ of level 1.  Thus
$\mathbb{RUNO}1$ denotes the subvariety of $\mathbb{RUNO}$ consisting of algebras from $\mathbb{RUNO}$ of level 1.
\end{Definition}
   
   The significance of the variety  $\mathbb{RUNO}1$ is revealed by the following theorem:
   
 \begin{Theorem}
 The algebras $\mathbf{A}i \in \mathbb{RUNO}1$, $i=1, 2, \cdots, 5$. 
 \end{Theorem}

\begin{Proof}
 It is routine to verify that the axioms of $\mathbb{RUNO}1$ are true in each $\mathbf{A}i$, $i=1, 2, \cdots, 5$.
 \end{Proof} 
 
 \medskip
 We wish to show that the converse of the above theorem holds as well.  To achieve this goal,  we need to characterize the subdirectly irreducible algebras in $\mathbb{RUNO}1$,  Toward this end, it will be helpful to know the maximum height of subdirectly irreducible algebras in $\mathbb{RUNO}1$.

\subsection{On the height of algebras of $\mathbb{RUNO}1$ satisfying (SC)} 

\
 
\indent  The purpose of this subsection is to show that the height of members of $\mathbb{RUNO}1$ that satisfy the condition (SC) (see below) is at most two--a result which will be useful in describing the subdirectly irreducible algebras in $\mathbb{RUNO}1$.
Let $x^+ := x'{^*}'$.

Unless stated otherwise, the lemmas below will have the following hypothesis:

``Let $\mathbf A \in \mathbb{RUNO}1$ satisfy:

{\rm(SC)} \quad $ x \neq 1 \Rightarrow x \land x'^* = 0.$

 Let $a := 0 \to 1$, and $b \in A$ such that $a< b<1$.''

\begin{Lemma} \label{125} 
 $b' \leq b$.  
\end{Lemma}

\begin{Proof}
From $a \leq b$ we have $b' \leq a'$, implying $b' \leq a$ as $a'=a$ by hypothesis, which yields $b' \leq b$, as $a \leq b$.  
\end{Proof}

\begin{Lemma} \label{182}
 $b' \lor  b'^*= b \lor b'^*$.    
\end{Lemma}

\begin{Proof}
As $b \neq 1$, $b \land b'^* = 0$ by (SC).  Hence,
 $b' \lor  b^+ =0'$, 
whence we have
\begin{equation} \label{225}
 b' \lor b^+ =1.     
\end{equation}
Now,
\begin{align*}
b' \lor b'^* &= (b' \lor  b'^*) \lor  (b \land b^+) &\text{ by regularity}\\
&=  (b' \lor  b'^* \lor  b) \land (b' \lor  b'^* \lor  b'{^*}')\\
&= (b'^* \lor  b) \land (b' \lor  b'^* \lor  b^+) &\text{ by Lemma \ref{125}}\\
&= b \lor b'^*  &\text{ by (\ref{225})},
\end{align*}
proving the lemma.
\end{Proof}

\begin{Lemma}  \label{200}
 $a \land  (b' \lor b'^*)  = b'.$  
\end{Lemma}

\begin{Proof}
Since $b \neq 1,$ we have $b \land b'^* = 0$ by (SC), which implies
  $a \land b'^* = 0,$ as $a \leq b,$  implying
\begin{equation}  \label{129} 
 a \land b'^*  = 0.  
\end{equation}
Hence we have
\begin{align*}
a \land  (b' \lor b'^*) &=(a \land b') \lor (a \land b'^*)\\
                               &= (a \land b') \lor 0 \\
                               &= a' \land b' \\
                               &= b'  &\text{ by  (\ref{129})},
 \end{align*}
 proving the lemma.
\end{Proof}

We are now ready to prove the following theorem which will be useful in the next section. 
\begin{Theorem} \label{Height_Th}
Let $\mathbf{A} \in \mathbb{RUNO}1$ satisfy:

{\rm(SC)} \quad $ x \neq 1 \Rightarrow x \land x'^* = 0.$\\
Then,
 the height of $\mathbf{A}$, $h(\mathbf{A}) \leq 2$.
\end{Theorem}

\begin{Proof} 
Since $a'=a$, it is clear that $a \neq 0$.  Now, suppose $h(\mathbf{A}) > 2$.  Then there exists an element $b \in A$ such that
  $0 < a < b < 1 $.   We wish to arrive at a contradiction.
Since $b' \lor b'^* = b \lor b'^*$ from Lemma \ref{182},  
we get from Lemma \ref{200} that
 $a \land (b \lor b'^*) = b',$  
implying $(a \land b) \lor (a \land b'^*) = b'.$  Therefore, in view of (\ref{129}), we obtain
 a  = b'. Hence $a=a'=b''=b$, which is a contradiction, completing the proof.
 \end{Proof}

\vspace{1cm} 
\section{A Characterization of Subdirectly irreducible Algebras in $\mathbb{RUNO}1$}

Our main goal in this section is to characterize the 
subdirectly irreducible algebras in $\mathbb{RUNO}1$.

The following theorem, which is a special case of a result proved in \cite{Sa14a}, is needed here. 

\begin{Theorem} {\rm \cite[Corollary 4.1]{Sa14a}}  \label{sub}
Let $\mathbf A \in \mathbb{DMSH}1$ with $|A| \geq 2$. Then the following are equivalent:
\begin{itemize}
\item[(1)] $ \mathbf A$ is simple; 
\item[(2)] $\mathbf A$ is subdirectly irreducible; 
\item[(3)] For every $x \in A$, if $x \neq 1$, then $x \land x'^* = 0$ \quad (SC).
\end{itemize}
\end{Theorem}

We are now ready to characterize the 
subdirectly irreducible algebras in $\mathbb{RUNO}1$.

\begin{Theorem} \label{SubIrr}
Let $\mathbf A \in \mathbb{RUNO}1$ with $|A| > 2$. Then the following are equivalent 
\begin{itemize}
\item[(1)] $ \mathbf A$ is simple 
\item[(2)] $\mathbf A$ is subdirectly irreducible 
\item[(3)] For every $x \in A$, if $x \neq 1$, then $x \land x'^* = 0$ \quad (SC). 
\item[(4)] $\mathbf A \in \{\mathbf A1, \mathbf A2, \mathbf A3, \mathbf A4, \mathbf A5\}$.
\end{itemize}
\end{Theorem}

\begin{Proof}
We know that (1), (2) and (3) are equivalent by Theorem \ref{sub}.  Now, suppose (3) holds in $\mathbf{A}$.  Then, by Theorem \ref{Height_Th} we know
the height is $\leq 2$.  Since $|A| > 2$, we conclude that the height of A =2.  Then it follows that $|A|=3$ or $|A|=4$.  If $|A|=3$ then it is easy to see that $\mathbf{A} \in \{ \mathbf{A}1, \mathbf{A}2, \mathbf{A}3, \mathbf{A}4\}$.  If $|A|=4$, then $\mathbf{A} =\mathbf{A}5$, up to isomorphism.  Thus, (4) holds in $\mathbf{A}$, whence (3) implies (4), while the statement that (4) implies (3) is routine to verify. Thus  the proof is complete. 
\end{Proof}

\begin{Corollary} \label{Runo1}
 $\mathbb{RUNO}1$ =  $\mathbb{V}(\{\mathbf A1, \mathbf A2, \mathbf A3, \mathbf A4, \mathbf A5\})$. 
\end{Corollary}

\begin{Remark} Thus, we have achieved an axiomatization of the variety 
$\mathbb{V}(\{\mathbf A1, \mathbf A2, \mathbf A3, \mathbf A4, \mathbf A5\})$,
thus solving {\bf PROBLEM C} posed earlier.
\end{Remark}

In the rest of this section we present some remarkable properties of the variety $\mathbb{RUNO}1$ as consequences of Corollary \ref{Runo1}.

\subsection{Discriminator}

\

\medskip

A ternary term $t(x, y, z)$ is a discriminator term for an algebra A if, for $a, b, c \in A$,
 \[
 t(a,b,c) = \begin{cases}
                c & \text{if } a= b \\
                a  & \text{if } a \neq b.
                 \end{cases}
\]

A variety $\mathbb{V}$ of algebras is a discriminator variety if there is a class $\mathbb{K}$ of algebras which generates $\mathbb{V}$ such that there is a ternary term $t(x, y, z)$ which is a discriminator term for every member of $\mathbb{K}$. 

\begin{Theorem} \label{Runo1 discriminator}
$\mathbb{RUNO}1$ is a discriminator variety. 
\end{Theorem}

\begin{Proof} 
Let $d_1(x) = x \land x'^*$.   
Let
$T(x, y, z)$
  $:=[z \land  d_1((x \lor y) \to (x \land y))] \lor [x \land \{d_1((x \lor y) \to (x \land y))\}^*]$.
 Then it is routine to verify that the term $T(x, y, z)$ is a ternary discriminator term on each of the algebras $\mathbf{Ai}$, {\bf i}=1,2, $\cdots$, 5.    
 Then the theorem follows from Corollary \ref{Runo1}.
  \end{Proof}
  
  The following theorem, due to Bulman-Fleming, Keimel, and
Werner. is well-known (See \cite[Page 165, Theorem 9.4]{BuSa81} for a proof).
  \begin{theorem} \label{Boolean Product}
  Every member of a discriminator variety is isomorphic to a Boolean product
of simple algebras from the variety. 
  \end{theorem}
  
    The following corollary is immediate from Theorem \ref{Boolean Product} and  
  Theorem \ref{Runo1 discriminator}.
  
  \begin{Corollary}\label{discriminator} 
  Every member of $\mathbb {RUNO}1$ is isomorphic to a Boolean product
of $\mathbf{Ai}, {\bf i}=1,2, \cdots, 5$. 
  \end{Corollary}
  
  \begin{Corollary}\label{Primal}
  The algebras $\mathbf{Ai}, {\bf i}=1,2, \cdots, 5$ are primal (just like the two element Boolean algebra).
  \end{Corollary}

   \begin{Corollary}\
   Every algebra in the variety $\mathbb{V}(\mathbf{Ai})$, ${\bf i} =1,2, \cdots, 5$ is a Boolean Power of $\mathbf{A}i$.
    \end{Corollary}
  
  \begin{Theorem} {\rm(Burris} \cite{Bu92}\rm)
  Every discriminator variety has unary unification type.
  \end{Theorem}
  
   \begin{Corollary}
   All subvarieties of $\mathbb {RUNO}1$  have unitary unification type. 
    \end{Corollary}
    
     It was shown in \cite{BuWe79} that the first-order theory of finitely generated discriminator variety is decidable.  Hence the following corollary is immediate.
  
\begin{Corollary}
 The first-order theory of each $\mathbb{V}(\mathbf{Ai}), {\bf i}=1,2, \cdots, 5$, is decidable.\\
\end{Corollary}

\subsection{Properties AP, SAP, ES}

\

\medskip
Amalgamation Property (AP, for short) is an important property for varieties to have.  We refer the reader to \cite{CzPi99} for its history and its connection to algebraic logic (see also \cite{CoSa24b}).

In this subsection we wish to investigate whether $\mathbb{V}(\mathbf{A}i)$, $i=1,2, \cdots, 5$ have AP, SAP and ES.
\begin{definition}
By a {\it diagram} in a class $\mathbb{K}$ of algebras we mean a quintuple $\langle A, f, B, g, C\rangle$ with $\mathbf A$, $\mathbf B$, $\mathbf C \in \mathbb K$ and $f : \mathbf A \mapsto \mathbf B$ and  $g : \mathbf A \mapsto \mathbf C$ embeddings. By an {\it amalgam}  of this diagram in $\mathbb K$ we mean a triple $\langle f_1, g_1, \mathbf D\rangle$ with $\mathbf D \in \mathbb K$ and with $f_1 : \mathbf B \mapsto \mathbf D$ and $g_1 : C \mapsto D$ embeddings such that $f_1\circ f= g_1\circ g$. If such an amalgam exists, then we say that the diagram is {\it amalgamable} in $\mathbb K$.  If such an amalgam exists and satisfies the condition:  $f_1(B) \cap g_1 (C)=f_1 circ f_(A)$, then we say that the diagram is {\it strongly amalgamable} in $\mathbb K$. 
We say that $\mathbb{K}$ has the Amalgamation Property (AP, for short) if every diagram in $\mathbb{K}$ is amalgamable.  $\mathbb{K}$ has the Strong Amalgamation Property (SAP, for short) if every diagram in $\mathbb{K}$ is strongly amalgamable.
\end{definition}

\begin{Remark}
Since f and g in a diagram are embeddings, we can, without loss of generality, simply think of a diagram $\langle A, f, B, g, C\rangle$ as a triple 
$\langle A, B, C\rangle$ such that $\mathbf A \subseteq \mathbf B \cap \mathbf C$.  Accordingly, the definition of (AP) can be restated.
\end{Remark}

For the rest of this subsection, we let $\mathbb{S}:= \{\mathbf {Ai}, {\bf i}=1,2, \cdots, 5\}$.
Note that $\mathbb{S}$ is prcisely the the class of simple, non-singleton members of $\mathbb {RUNO}1$.

\begin{Lemma} \label{APLemma}
Every non-empty subclass of $\mathbb{S}$ has AP.
\end{Lemma}

\begin{proof}
Let $\mathbb{S'} \subset \mathbb{S}$ be nonempty.  Recall that every member of $\mathbb{S'} $ has no proper subalgebras.  Hence the only possible diagrams are those in which
all three algebras are equal to each other; and such diagrams can easily be amalgamated, which implies that $\mathbb{S'}$ has AP.
\end{proof}
\begin{Theorem} {\rm(Werner \cite[page 27]{We78}\rm)}  \label{AP}
A discriminator variety has the amalgamation property iff the class of its simple, non-singleton members has the amalgamation property.
\end{Theorem}

 The following corollary is immediate from Theorem \ref{Runo1 discriminator}, Lemma \ref{APLemma} and Theorem \ref{AP}.
 
\begin{Corollary} \label{C1}
Every subvariety of $\mathbb {RUNO}1$ has the amalgamation property. 
 \end{Corollary}

 \begin{definition} [Epimorphism] Let $\mathbb{K}$ be a class of algebras, and $\mathbf{A}, \mathbf{B} \in \mathbb{K}$.
A homomorphism $h : A \to B$ is called a $\mathbb{K}$-epimorphism, or K-epi, if it satisfies the
following condition: For any $\mathbf{C} \in \mathbb{K}$ and any pair of homomorphisms $k, k': \mathbf{B} \to
\mathbf{C}$, if $k \circ h = k' \circ h$, then $k = k'$.
We say that $\mathbb{K}$ has the property: Epimorphisms are surjective (ES, for short) if every epimorphism on every algebra in $\mathbb{K}$ is surjective (i.e., onto).
\end{definition}

Since all members of $\mathbb{S}$ are primal, the following lemma is immediate.

\begin{Lemma} \label{ESLemma}
Every non-empty subclass of $\mathbb{S}$ has ES.
\end{Lemma}

\begin{theorem} Comer \cite{Co69} \label{ES}
Let V be a discriminator variety such that the
class of its simple members has AP and ES. Then V satisfies ES.
\end{theorem}

\begin{Corollary} \label{C2}.  
Every subvariety of $\mathbb {RUNO}1$ satisfies ES.
 \end{Corollary}
 
 \begin{proof}
 Let $\mathbb{V}$ be a subvariety of $\mathbb{RUNO}1$. 
 Recall that the variety $\mathbb{V}$ 
 is a discriminator variety.  Also, $\mathbb{S}$ has AP by Lemma \ref{APLemma} and has
ES by Lemma \ref{ESLemma}.  Hence, we conclude, in view of Theorem \ref{ES} that $\mathbb{V}$ has ES. 
\end{proof}

\begin{theorem} Isbell \cite{Is66} \label{SAP} Let V be a variety.  
If V has ES and AP, then V has SAP.
\end{theorem}

\begin{Corollary}
Every subvariety of $\mathbb {RUNO}1$ satisfies SAP.  
 \end{Corollary}

 \begin{proof}
Let $\mathbb{V}$ be a subvariety of $\mathbb{RUNO}1$.  Then $\mathbb{V}$ has AP by Corollary \ref{C1} and has ES by Corollary \ref{C2}. Hence by Theorem \ref{SAP} we conclude that $\mathbb{V}$ has SAP.
\end{proof}

\section{{\bf The lattice of subvarieties of $\mathbb{RUNO}1$}}

In this section we present a concrete description of the structure of the lattice of subvarieties of  $\mathbb{RUNO}1$ and also give an (equational) base for each of the thirty nontrivial, proper subvarieties of $\mathbb{RUNO}1$.

\begin{Corollary} \label{CorLat}
The lattice $\mathbf{L}$ of subvarieties of the variety $\mathbb{RUNO}1$ is 
a Boolean lattice of size $2^5$ \rm{(}=32\rm{)} elements. 
\end{Corollary}

\begin{Proof} Since, by Corollary \ref{Runo1},  $\mathbb{RUNO}1$ is generated by $\{\mathbf{Ai}: i=1, \cdots, 5\}$, and all the generators are primal by Corollary \ref{Primal}, it follows that the five generators are atoms of the distributive lattice $\mathbf{L}$.  Then, using J\'{o}nsson's theorem \cite{Jo67}, we can conclude that 
\begin{itemize}
\item of height 0: The trivial variety;
\item of height 1:  $\mathbb{V}(\mathbf{A}i)$, $i=1,2, \cdots,5$ as atoms;
\item of height 2: subvarieties generated by pairs; 
\item of height 3: subvarieties generated by (unordered) triples; 
\item of height 4: subvarieties generated by (unordered) quadruples of these algebras; 
\item of height 5: the variety $\mathbb{RUNO}1$. 
 \end{itemize}
 It is clear that the lattice $\mathbf{L}$ is complemented and hence is Boolean.  Thus we  conclude that $\mathbf{L}$ is isomorphic to the power set of $\{1,2,3,4,5\}$; hence $\mathbf{L}$ is
a Boolean lattice of size $2^5$ (= 32) elements. 
\end{Proof}

\subsection{{\bf Bases for subvarieties of $\mathbb{RUNO}1$}}

\

\medskip
In this section, we will give bases for all the subvarieties of $\mathbb{RUNO}1$.  
Here ``a base for a class $\mathbb{K}$ of algebras'' means ``a base for the variety 
$\mathbb{V(K)}$ of algebras.''  In particular, ``a base for a single algebra $\mathbf A$'' means 
a base for $\mathbb{V}(\mathbf{A})$.  Also, a ``base`` in the following theorem means ``a base relative to $\mathbb{RUNO}1$.  Note also that $a := 0 \to 1$, $b := 0 \to a.$

\begin{Theorem}
Bases for all non-trivial and proper subvarieties of $\mathbb{RUNO}1$ are as follows:
\smallskip
\begin{thlist}
\item[1] {\bf BASES FOR SINGLE ALGEBRAS :}
          
\item[i] \quad {\bf A BASE FOR $\{\mathbf A1\}$}  
     \begin{thlist}
\item[a]  $(0 \to 1) \to 1 \approx 1$; 
 \item[b]  $0 \to (0 \to 1) \approx 1.$\\ 
      \end{thlist}
\item[ii] \quad  {\bf A BASE FOR $\{\mathbf A2\}$: } 
      \begin{thlist}
\item[a]  $0 \to (0 \to 1) \approx 1$;
\item[b] $(0 \to 1) \to 1 \approx 0 \to1$;
\item[c] $(0 \to 1)^*   \approx 0$.\\
         \end{thlist}
\item[iii] \quad {\bf A BASE FOR $\{\mathbf A3\}$ : } 
      \begin{thlist}
\item[a] $x \to (x \to (x \to y))  \approx  x \to (x \to y)$.\\
        \end{thlist}
\item[iv] \quad {\bf A BASE FOR $\{\mathbf A4\}$ :} 
        \begin{thlist}
\item[a]  $((0 \to 1) \to 1)' \approx 0 \to (0 \to 1)$.\\
        \end{thlist}
\item[v] \quad {\bf A BASE FOR $\{\mathbf A5\}$ : } 
       \begin{thlist}
\item[a]  $(0 \to 1)^* \to 1 \approx 1$.\\
       \end{thlist}

\item[2] {\bf BASES FOR PAIRS OF ALGEBRAS :}

\smallskip
    
\item[i] \quad {\bf A BASE FOR $\{\mathbf A1, \mathbf A2\}$ : } 
\begin{thlist}
\item[a] $(0 \to 1)^* \approx 0$;
\item[b] $(0 \to (0 \to 1) \approx 1.$\\
 \end{thlist}
 
\item[ii] \quad {\bf A BASE FOR $\{\mathbf A1, \mathbf A3\}$ : } 
\begin{thlist}
\item[a] $(0 \to 1)^* \approx 0$;
\item[b] $ (0 \to 1) \to 1 \approx 1.$\\
\end{thlist}

\item[iii] \quad {\bf A BASE FOR $\{\mathbf A1, \mathbf A4\}$ : } 
\begin{thlist}
\item[a] $(0 \to 1)^* \approx 0$;

\item[b] $(0 \to (0 \to 1)) \approx ((0 \to 1) \to 1)$.\\
  \end{thlist}

\item[iv] \quad {\bf A BASE FOR $\{\mathbf A1, \mathbf A5\}$ : } 
     \begin{thlist}
\item[a] $(0 \to (0 \to 1)) \to (0 \to 1)^{**}  \approx  (0 \to 1)^{**};$

\item[b] $(0 \to (0 \to 1)) \geq ((0 \to 1) \to 1)$;
 
\item[c] $(0 \to 1)^* \to ((0 \to 1) \to 1) \approx (0 \to 1).$\\
    \end{thlist}

\item[v] \quad {\bf A BASE FOR $\{\mathbf A2, \mathbf A3\}$: } 
  \begin{thlist}
\item[a] $(0 \to 1)^* \approx 0$;

 \item[b] $(0 \to (0 \to 1)) \to (0 \to 1)^{**}  \approx  (0 \to 1)^{**};$
 
\item[c] $(a \to b) \to a \approx b.$\\  
\end{thlist}

\item[vi] \quad {\bf A BASE FOR $\{\mathbf A2, \mathbf A4\}$ :  } 
 \begin{thlist}
\item[a] $(0 \to 1)^* \approx 0$;

\item[b] $(0 \to (0 \to 1)) \geq ((0 \to 1) \to 1)$;
 
\item[c]  $(a \to b) \to a \approx b.$\\ 
 \end{thlist}

\item[vii] \quad {\bf A BASE FOR $\{\mathbf A2, \mathbf A5\}$ : } 
    \begin{thlist}
  \item[a] $(0 \to 1) \to (0 \to (0 \to 1)) \approx  0 \to 1;$
  
   \item[b] $(0 \to (0 \to 1)) \geq ((0 \to 1) \to 1)$;
   
  \item[c] $(a \to b) \to a \approx b.$ \\ 
    \end{thlist}

\item[viii] \quad {\bf A BASE FOR $\{\mathbf A3, \mathbf A4\}$: } 
 \begin{thlist}
\item[a] $(0 \to 1)^* \approx 0$;

\item[b] $(0 \to (0 \to 1) \approx 0 \to 1;$
 
\item[c] $(a \to b) \to a \approx b.$\\ 
 \end{thlist}

\item[ix] \quad {\bf A BASE FOR $\{\mathbf A3, \mathbf A5\}$ } 
 \begin{thlist}
 \item[a] $(0 \to (0 \to 1)) \to (0 \to 1)^{**}  \approx  (0 \to 1)^{**};$
 
 \item[b] $(0 \to 1)^* \to ((0 \to 1) \to 1) \approx (0 \to 1);$
 
\item[c] $(a \to b) \to a \approx b.$\\  
 \end{thlist}

\item[x] \quad {\bf A BASE FOR $\{\mathbf A4, \mathbf A5\}$ } 
 \begin{thlist}
\item[a] $(0 \to (0 \to 1)) \geq ((0 \to 1) \to 1)$;
  
\item[b] $(0 \to 1)^* \to ((0 \to 1) \to 1) \approx (0 \to 1);$
   
 \item[c] $(a \to b) \to a \approx b.$\\  
 \end{thlist} 

\item[3] {\bf BASES FOR TRIPLES OF ALGEBRAS :}

\item[i] \quad {\bf A BASE FOR $\{\mathbf A1, \mathbf A2, \mathbf A3\}$ : } 
\begin{thlist}
\item[a] $(0 \to 1)^* \approx 0$;

\item[b] $(0 \to (0 \to 1)) \to (0 \to 1)^{**}  \approx  (0 \to 1)^{**}.$\\
\end{thlist}

\item[ii] \quad {\bf A BASE FOR $\{\mathbf A1, \mathbf A2, \mathbf A4\}$ : } 
\begin{thlist}

\item[a]  $(0 \to 1)^* \approx 0$;

\item[b] $(0 \to (0 \to 1)) \geq ((0 \to 1) \to 1)$.\\
\end{thlist}

\item[iii] \quad {\bf A BASE FOR $\{\mathbf A1, \mathbf A2, \mathbf A5\}$ : } 
\begin{thlist}
\item[a] $(0 \to (0 \to 1)) \to (0 \to1)^{**} \approx  (0 \to 1)^{**}$;

\item[b] $(0 \to (0 \to 1)) \geq ((0 \to 1) \to 1)$.\\ 
\end{thlist}

\item[iv] \quad {\bf A BASE FOR $\{\mathbf A1, \mathbf A3, \mathbf A4 \}$: } 
\begin{thlist}
\item[a] $(0 \to 1)^* \approx 0$;

\item[b] $0  \to (0 \to (0 \to 1) \approx 0 \to 1;$

 \item[c] $(0 \to 1)^* \to ((0 \to 1) \to 1) \approx (0 \to 1).$\\
\end{thlist}

\item[v] \quad {\bf A BASE FOR $\{\mathbf A1, \mathbf A3, \mathbf A5\}$ : } 
\begin{thlist}
\item[a] $(0 \to 1)^* \to ((0 \to 1) \to 1) \approx (0 \to 1)$;
 
\item[b] $(0 \to (0 \to 1)) \to (0 \to 1)^{**}  \approx  (0 \to 1)^{**}.$\\
\end{thlist}

\item[vi] \quad {\bf A BASE FOR $\{\mathbf A1, \mathbf A4, \mathbf A5 \}$: } 
\begin{thlist} 
\item[a] $(0 \to (0 \to 1)) \geq ((0 \to 1) \to 1)$;
 
\item[b] $(0 \to 1)^* \to ((0 \to 1) \to 1) \approx (0 \to 1).$\\
\end{thlist}

\item[vii] \quad {\bf A BASE FOR $\{\mathbf A2, \mathbf A3, \mathbf A4\}$ :} 
\begin{thlist} 
\item[a] $(0 \to 1)^* \approx 0$;

\item[b] $(a \to b) \to a \approx b.$\\ 
\end{thlist}

\item[viii] \quad {\bf A BASE FOR $\{\mathbf A2, \mathbf A3, \mathbf A5\}$ : } 
\begin{thlist}
\item[a] $(0 \to (0 \to 1)) \to (0 \to 1)^{**}  \approx  (0 \to 1)^{**}$;
 
\item[b] $(a \to b) \to a \approx b.$ \\ 
\end{thlist}

\item[ix]\quad {\bf A BASE FOR $\{\mathbf A2, \mathbf A4, \mathbf A5 \}$:} 
\begin{thlist}
 \item[a] $(0 \to (0 \to 1)) \geq ((0 \to 1) \to 1)$; 
  
\item[b] $(a \to b) \to a \approx b.$\\ 
\end{thlist}

\item[x] \quad {\bf A BASE FOR $\{\mathbf A3, \mathbf A4, \mathbf A5\}$: } 
  \begin{thlist}
\item[a] $(0 \to 1)^* \to ((0 \to 1) \to 1) \approx (0 \to 1)$;
 
\item[b] $(a \to b) \to a \approx b.$\\ 
      \end{thlist}

\item[4] {\bf BASES FOR QUADRUPLES OF ALGEBRAS :}
\vspace{.3cm}
\item[i] \quad {\bf A BASE FOR $\{\mathbf A1, \mathbf A2, \mathbf A3, \mathbf A4\}$: } 
\begin{thlist}
\item[a] $(0 \to 1)^* \approx 0$.\\
\end{thlist}

\item[ii] \quad {\bf A BASE FOR $\{\mathbf A1, \mathbf A2, \mathbf A3, \mathbf A5\}$: } 
   \begin{thlist}
\item[a] $(0 \to (0 \to 1)) \to (0 \to 1)^{**}  \approx  (0 \to 1)^{**}.$\\
 \end{thlist}

\item[iii] \quad {\bf A BASE FOR $\{\mathbf A1, \mathbf A2, \mathbf A4, \mathbf A5\}$: } 
      \begin{thlist}
 \item[a] $(0 \to (0 \to 1)) \geq ((0 \to 1) \to 1)$.\\ 
       \end{thlist}

\item[iv] \quad {\bf A BASE FOR $\{\mathbf A1, \mathbf A3, \mathbf A4, \mathbf A5\}$: } 
      \begin{thlist}
\item[a]  $(0 \to 1)^* \to ((0 \to 1) \to 1) \approx (0 \to 1).$\\
       \end{thlist}

\item[v] \quad {\bf A BASE FOR $\{\mathbf A2, \mathbf A3, \mathbf A4, \mathbf A5\}$: } 
       \begin{thlist}
\item[a] $(a \to b) \to a \approx b.$\\ 
       \end{thlist}
\end{thlist}
\end{Theorem}

\begin{Proof} We illustrate the proofs by proving (1)(i): Routine verification shows that the algebra $\mathbf A1$ satisfies (a) and (b), while both (a) and (b) fail in the remaining four algebras , proving that (a) and (b) form an (equational) base for $\mathbb{V}(\mathbf{A})$.  Similar arguments apply for the remaining parts of the theorem.
\end{Proof}

\begin{Corollary}
Let $\mathbb{A} \in \mathbb{RUNO}1$ such that
$\mathbb{A} \models (0 \to 1) \to1 \approx 1.$ 
Then
$\mathbb{A} \models (0 \to 1)^* \approx 0.$
\end{Corollary}

\begin{Proof}
Let $\mathbf{A}$ be as in the hypothesis.  Then it follows from the preceding theorem that $\mathbf A \in \mathbb{V}(\mathbf{A1, A3})$.  Observe that the algebras $\mathbf{A}1$ and $\mathbf{A}3$ satisfy the conclusion.  So, the conclusion holds in $\mathbf A$.   
\end{Proof}

\section{More Identities}

\smallskip
We will give several more consequences to Theorem \ref{sub} in the rest of this section.
The following corollaries to Theorem \ref{sub} (or to Corollary \ref{Runo1}) are immediate.
Recall that $a := 0 \to 1$, $b := 0 \to a$.  Also, let $c := 0 \to b$.

\begin{Corollary} 
\begin{thlist}

\item[1] $\mathbb{RUNO}1 \models x^* \lor x^{**} \approx 1$. 

\item[2] $\mathbb{RUNO}1 \models x'{^*}'{^*}'{^*} \approx  x'{^*}$.  

\item[3] $\mathbb{RUNO}1 \models x{^*}'{^*}'{^*}' \approx  x{^*}'$.

\item[4] $\mathbb{RUNO}1  \models  [((x \to y) \to y) \to y ] \to y \approx (x \to y) \to y.$  

\item[5] $\mathbb{RUNO}1  \models  x \to [x \to (x \to (x \to y))] \approx x \to (x \to y).$ 

\item[6] $\mathbb{RUNO}1 \models 0 \to (0 \to (0 \to 1)) \approx (0 \to 1)$.  

\item[7] $\mathbb{RUNO}1 \models (0 \to (0 \to 1))^* \approx 0.$        
 
\item[8] $\mathbb{RUNO}1 \models (0 \to 1){^*}' \to (0 \to 1)\approx  0 \to 1$.

\item[9] $\mathbb{RUNO}1 \models x^* \lor (y \to z) \approx  (x^* \lor y) \to (x^* \lor z)$ (JID*).
.

\item[10] $\mathbb{RUNO}1 \models ((0 \to 1)\to 1) \to 1 \approx (0 \to 1) \to1$.  

\item[11] $\mathbb{RUNO}1 \models  c \approx a$.

\item[12]  $\mathbb{RUNO}1 \models 0 \to 1 \ \leq \ 0 \to (0 \to 1).$ 

\item[13]  $\mathbb{RUNO}1 \models x^* \to (y^* \to z^*) \approx  y^* \to (x^* \to z^*)$  (Weak Exchange).

\item[14]  $\mathbb{RUNO}1 \models (0 \to1) \to (0\to 1)^* \approx  0.$ 
 \end{thlist}
 \end{Corollary}

\begin{corollary}
\begin{thlist}
\item[a]  $\mathbb{V}(\mathbf{A1, A2, A3, A4}) \models (0 \to 1)^* \approx 0$. 

\item[b]   $\mathbb{V}(\mathbf{A1, A2, A3, A4}) \models x{^*}' \approx  x^{**}$.

\item[c]  $\mathbb{V}(\mathbf{A1, A2, A3, A4}) \models  [(0 \to 1) \to 1]^{**} \approx 1.$

\item[d]  $\mathbb{V}(\mathbf{A1, A2, A3, A4}) \models (0 \to 1) \to [0 \to (0 \to 1)]^* \approx 
[0 \to (0 \to 1)]^*.$

\item[e] $\mathbb{V}(\mathbf{A1, A2, A3, A4}) \models  [0 \to (0 \to 1)]^* \to (0 \to 1)^{**} \approx 0 \to 1.$

\item[f]   $\mathbb{V}(\mathbf{A1, A2, A3, A4}) \models x'{^*}'{^*} \approx  x'{^*}$.

 \item[g]  $\mathbb{V}(\mathbf{A1, A2, A3, A4}) \models (0 \to1) \to (0\to 1)^* \approx  (0 \to 1)^*.$ 
 
 \item[h]  $\mathbb{V}(\mathbf{A1, A2, A3, A4}) \models (0 \to 1){^*}'^* \approx (0 \to 1)^*$.  
 
 \item[i]  $\mathbb{V}(\mathbf{A1, A2, A3, A4}) \models [0 \to (0 \to (0 \to 1))]'^*  \to (0 \to (0 \to 1)) \approx  0 \to 1$.
 
 \item[j] $\mathbb{V}(\mathbf{A1, A2, A3, A4}) \models 0 \to (0 \to 1)^*  \approx  1.$ 
 
 \item[k] $\mathbb{V}(\mathbf{A1, A2, A3, A5}) \models [0 \to (0 \to 1)] \to (0 \to 1)^{**}  \approx  (0 \to 1)^{**}.$

\item[l]  $\mathbb{V}(\mathbf{A1, A2, A4, A5}) \models [0 \to (0 \to 1)] \geq [(0 \to 1) \to 1]$. 

\item[m] $\mathbb{V}(\mathbf{A1, A3, A4,A5}) \models (0 \to 1)^* \to [(0 \to 1) \to 1] \approx 0 \to 1.$ 

\item[n]  $\mathbb{V}(\mathbf{A2, A3, A4, A5}) \models (a \to b) \to b \approx a$.

\item[o]  $\mathbb{V}(\mathbf{A2, A3, A4, A5}) \models (a \to b) \approx b \to a.$

\item[p]  $\mathbb{V}(\mathbf{A2, A3, A4, A5}) \models (a \to b) \to a \approx b.$

\end{thlist}
\end{corollary}

\begin{corollary}
\begin{thlist}
\item[a]  $\mathbb{V}(\mathbf{A1, A2, A5}) \models 0 \to (0 \to 1) \approx 1$. 

\item[b]  $\mathbb{V}(\mathbf{A1, A2, A5}) \models x \to [x \to (x \to y)] \approx x \to y$. 

\item[c]  $\mathbb{V}(\mathbf{A1}, \mathbf{A2}, \mathbf{A5}) \models 0 \to (0 \to 1) \approx 1$. 

\item[d] $\mathbb{V}(\mathbf{A1, A3, A4}) \models (a \to b) \to a  \approx  a$. 

\item[e]  $\mathbb{V}(\mathbf{A1, A3, A4}) \models [0 \to (0 \to 1)] \leq [(0 \to 1) \to 1]$. %

\item[f]  $\mathbb{V}(\mathbf{A1, A3, A4}) \models (0 \to 1) \to [0 \to (0 \to 1)] \approx 1.$ 

\item[g]  $\mathbb{V}(\mathbf{A2, A4, A5}) \models (0 \to 1) \to 1\approx 0 \to 1.$ 

\item[h] $\mathbb{V}(\mathbf{A2,A4,A5}) \models [(x \to y) \to y] \to y \approx x \to y.$ 

\item[i] $\mathbb{V}(\mathbf{A2, A4, A5}) \models 0 \to [0 \to (0 \to 1)] \approx (0 \to 1) \to 1$. 

\item[j]  $\mathbb{V}(\mathbf{A2, A4, A5}) \models 0 \to [0 \to (0 \to 1)] \approx [(0 \to 1) \to 1] \to 1$.

\item[k]  $\mathbb{V}(\mathbf{A3, A4, A5}) \models (0 \to1){^*} \to (0\to 1) \approx  0 \to 1.$ 

\item[l]  $\mathbb{V}(\mathbf{A3, A4, A5}) \models (0 \to 1)^* \to [0 \to (0 \to 1)] \approx  0 \to (0 \to 1)$. 
\end{thlist}
\end{corollary}

\begin{corollary}
\begin{thlist}
 
\item[1]  $\mathbb{V}(\mathbf{A1, A3}) \models (0 \to 1) \to 1 \approx 1.$  
 
\item[2]   $\mathbb{V}(\mathbf{A1, A4}) \models [0 \to (0 \to 1)] \approx (0 \to 1) \to 1.$ 

\item[3]   $\mathbb{V}(\mathbf{A2, A5}) \models (0 \to 1) \to [0 \to (0 \to 1)] \approx  0 \to 1.$  

\item[4]  $\mathbb{V}(\mathbf{A3}, \mathbf{A4}) \models 0 \to (0 \to 1)\approx  0 \to 1.$ 
 \end{thlist}
\end{corollary}

\begin{corollary} \label{corA}
\begin{thlist}
\item[a]  $\mathbb{V}(\mathbf{A1}) \models (0 \to 1) \to [0 \to (0 \to 1)] \approx  0 \to (0 \to 1)$. 

\item[b]  $\mathbb{V}(\mathbf{A3}) \models (0 \to 1) \to 1 \approx 1.$

\item[c] $\mathbb{V}(\mathbf{A3}) \models 0 \to (0 \to 1) \approx  0 \to 1.$

\item[d]  $\mathbb{V}(\mathbf{A5}) \models (0 \to1){^*}' \approx  (0 \to 1)^*.$ 

\item[e]  $\mathbb{V}(\mathbf{A5}) \models (0 \to 1){^*}^* \approx (0 \to 1)$.

\item[f]  $\mathbb{V}(\mathbf{A5}) \models 0 \to [0 \to (0 \to1)] \approx  0 \to (0 \to 1)^*.$ 

\item[g]  $\mathbb{V}(\mathbf{A5}) \models (0 \to 1)^* \to 1 \approx 1$.

\item[h]  $\mathbb{V}(\mathbf{A5}) \models x \to (y \to z) \approx  y \to (x \to z)$  (Exchange). 

\item[i]  $\mathbb{V}(\mathbf{A5}) \models x' \lor (y \to z) \approx  (x' \lor y) \to (x' \lor z)$  (JID).

\item[j]  $\mathbb{V}(\mathbf{A5}) \models [(0 \to 1) \to 1]{^*}^* \approx  0 \to 1.$ 

\item[k]  $\mathbb{V}(\mathbf{A5}) \models (0 \to 1) \lor (0 \to 1)^* \approx 1$. 
\end{thlist}
\end{corollary}

\begin{Remark}
Each of the equations in Corollary \ref{corA} gives a 1-base for the variety generated by the corresponding algebra.
\end{Remark}

\section{Toward a logic for the variety $\mathbb{RUNO}1$}

In this section we introduce and investigate a propositional logic, in Hilbert-style, called  
$\mathcal{RUNO}1$. 
It will be shown that $\mathcal{RUNO}1$ is
 implicative in the sense of Rasiowa, and hence is
algebraizable (in the sense of Blok and Pigozzi) with
 the variety $\mathbb {RUNO}1$ of regular unorthodox algebras of level 1 as its equivalent algebraic semantics.

\subsection{Abstract Algebraic Logic:  Basic definitions and results} 
 
\

\medskip
In this subsection, we recall the basic definitions and results of Abstract Algebraic Logic that will play a crucial role. 

\begin{itemize}
\item Let us fix the  language $\mathbf L = \langle \lor, \land, \to, \sim,  \bot, \top\rangle$.

\item The set of {\bf formulas} in {\bf L} will be denoted by $\mbox{\it Fm}_{\,\bf L}$.   

\item The set of formulas $\mbox{\it Fm}_{\,\bf L}$  can be turned into an {algebra of formulas}, denoted by ${\bf Fm}_{\bf L}$, in the usual way. 

\item  $\Gamma$ denotes a set of formulas and lower case Greek letters denote single formulas.  

\item The homomorphisms from the formula algebra ${\bf Fm}_{\bf L}$ into an $\mathbf L$-algebra (i.e, an algebra of type $\mathbf L$) $\mathbf A$ are called {\bf interpretations} (or {\bf valuations}) in A. The set of all such interpretations is denoted by  $Hom({\bf Fm}_{\bf L},\mathbf A)$. 

\item If $h \in Hom({\bf Fm}_{\bf L},\mathbf A)$ then the {\em interpretation of a formula} $\alpha$ under $h$ is its image $h( \alpha) \in A$. 
\item  $h[\Gamma]$ denotes the set $\{h\phi \ | \ \phi \in \Gamma\}$.  
\end{itemize}

\newpage
{\bf Consequence Relations and Logics}: 
\begin{itemize}
\item  A {\bf consequence relation} on $Fm_{\mathbf L}$ is a binary relation $\vdash$ between sets of
formulas and formulas that satisfies the following conditions:

 for all $\Gamma$, $\Delta  \subseteq Fm_{\mathbf L}$ and $\phi \in Fm_{\bf L}$ :

(i) $\phi \in \Gamma$ implies $\Gamma \vdash \phi,$

(ii) $ \Gamma \vdash \phi$ and $\Gamma \subseteq \Delta$ imply $\Delta \vdash \phi,$

(iii) $\Gamma \vdash \phi$ and $\Delta \vdash \beta$ for every $\beta \in \Gamma$ imply $\Delta \vdash \phi.$\\

\item A consequence relation $\vdash$ is {\bf finitary} if
$\Gamma \vdash \phi$ implies $\Gamma' \vdash \phi$ for some finite $\Gamma' \subseteq \Gamma$. \\

\item  
A consequence relation $\vdash$ is {\bf structural} if
$\Gamma \vdash \phi$  implies $\sigma[\Gamma] \vdash \sigma(\phi)$ for every substitution $\sigma$ 
$\in Hom(\mathbf{Fm_L}, \mathbf{Fm_L})$.   ($\sigma[\Gamma] :=\{\sigma(\alpha) : \alpha \in \Gamma\}$). \\

\item 
A {\bf logic} (or {\bf deductive system}) is a pair $\mathcal L := \langle \mathbf{L}, \vdash_{\mathcal L} \rangle$, where $\mathbf{L}$ is a propositional
language and $\vdash_{\mathcal L}$ is a finitary, structural consequence relation on Fm$_{\mathbf L}$.

\item {\bf Logics in a Hilbert-style}:

One way to present a logic $\mathcal L$ is by displaying it (syntactically) in {\bf Hilbert-style}; that is, giving its axioms and rules of inference which induce a
 consequence relation $\vdash_{\mathcal L}$ as follows:\\
\indent  $\Gamma \vdash_{\mathcal L} \phi$ if there is  a {\bf (formal) proof} (or, a {\bf derivation}) of $\phi$ from $\Gamma$, where a proof is defined as a 
 sequence of formulas $\phi_1, \dots, \phi_n$, $n \in \mathbb N$, such that
$\phi_n = \phi$, and for every $i \leq n$, one of the following conditions holds:

(i) $\phi_i  \in \Gamma$,

(ii) there is an axiom $\psi$  and a substitution $\sigma$ such that $\phi_i = \sigma \psi $, 

(iii) there is a rule $\langle \Delta, \psi \rangle$ and a substitution $\sigma$  such that $\phi_i  = \sigma\psi$  and
$\sigma[\Delta] \subseteq \{\phi_j : j < i\}$. 
 
\item  Identities in $\mathbf L$ are ordered pairs of $\mathbf L$-formulas that will be written in the form $\alpha \approx \beta$.  

\item If $\bar{x}$ is a sequence of variables and $h$ is an interpretation in $\mathbf{A}$, then we write $\bar{a}$ for $h(\bar{x})$.

\item $\mathbf{A} \models \alpha \approx \beta \quad \mbox{ if and only if } \quad \alpha^{\mathbf A} (\bar{a}) = \beta^{\mathbf A} (\bar{a}),  
\mbox{ for all } \bar{a} \in \mathbf A.$  

\item For a class $\mathbb K$ of algebras of type {\bf L},  
$\mathbb{K} \models \alpha \approx \beta \mbox{ if and only if } \mathbf{A}\models \alpha \approx \beta, \mbox{ for all }\mathbf A \in \mathbb K.$ 

\item
We define the relation $\models_{\mathbb K}$ between a set $\Gamma$ of identities and an identity $\alpha \approx \beta$ as follows: 

\quad $\Gamma \models_{\mathbb K} \alpha \approx \beta \quad \text{ if and only if }$ \quad
 for every $\mathbf A \in \mathbb{K}$ and for every interpretation $\bar{a}$ of the 
 variables of $\Gamma \cup  \{\alpha \approx \beta\}$ in $\mathbf{A}$, \\
{\bf if $\phi^{\mathbf A}(\bar{a} ) = \psi^{\mathbf A} (\bar{a}) $,  for every $\phi \approx \psi \in \Gamma$, then \\ 
\quad $\alpha^{\mathbf A}(\bar{a}) = \beta^{\mathbf A}(\bar{a})$}.

\item
The relation $\models_{\mathbb{K}}$ is called the {\bf (semantic) equational consequence relation determined by $\mathbb K$}.

\item
In this case, we say that $\alpha \approx \beta$ is a {$\mathbb K$-consequence} of $\Gamma$. \\
\end{itemize}

\subsection{{\bf Algebraizability of a logic}} 

\

\medskip
\noindent {\bf Algebraic Semantics for a Logic:}

Let $\langle L, \vdash_L\rangle$ be a logic and K a class of $L$-algebras. K is called an {\bf algebraic semantics} for $\langle L, \vdash_L\rangle$ 
if $\vdash_L$ can be interpreted in $\vdash_K$ in the following sense:

There exists a finite set $\delta_i (p) \approx \epsilon_i (p)$, for $i < n$, of identities with a single variable $p$ such that, for all $\Gamma \cup \phi \subseteq Fm$ and each $ j < n$, 
$$ \quad \Gamma \vdash_L \phi \Leftrightarrow  \{\delta_i [\psi /p]  \approx \epsilon_i [\psi /p], i < n, \psi \in \Gamma\} 
\models_K \delta_j [\phi /p] \approx \epsilon_j [\phi /p],$$ 

\indent where $\delta[\psi / p]$  denotes the formula obtained by the substitution of $\psi$ at every occurrence of $p$ in $\delta. $

 The identities $\delta_i \approx \epsilon_i$, for $i < n$, are called {\bf defining identities} for $\langle L, \vdash_L\rangle$ and $\mathbb K$.

\medskip
\noindent {\bf Equivalent Algebraic Semantics and Algebraizable Logic}:

\begin{Definition}  Let $\mathcal S$ be a logic over a language {\bf L} and $\mathbb K$ an algebraic semantics
of S with defining equations $\delta_i(p) \approx \epsilon_i(p)$, $i \leq n$.  Then, $\mathbf K$ is an {\bf equivalent
algebraic semantics} of $\mathbf{S}$ if there exists a {\bf finite} set $\{\Delta_j (p, q) : j \leq m\}$ of
formulas in two variables satisfying the condition:

For every  $\phi  \approx \psi \in Eq_{\mathbf L}$,

$\phi  \approx \psi 
            \models_K \{\delta_i(\Delta_j (\phi, \psi)) \approx \epsilon_i(\Delta_j (\phi, \psi)) : i \leq n, j \leq m\}$ and

$\{\delta_i(\Delta_j (\phi, \psi)) \approx \epsilon_i(\Delta_j (\phi, \psi)) : i \leq n, j \leq m\} \ \models_K  \phi  \approx \psi$.\\

\noindent The set $\{\Delta_j (p, q) : j \leq m \}$ is called an {\bf equivalence system}. \\

A logic is {\bf algebraizable (in the sense of Blok and Pigozzi)} if it has an equivalent algebraic semantics.
\end{Definition}

\subsection{Algebraizabilty of Implicative Logics}:
\

\medskip

The following definition  is well known (see \cite[page 179]{Ra74}). 
\begin{definition} {\rm \cite{Ra74}}
	Let $\mathcal L$  be a logic in a language that includes a binary connective $\to$, either primitive or defined by a term in exactly two variables. Then $\mathcal L$ is called an implicative logic with respect to the binary connective $\to$ if the following conditions are satisfied:
	\begin{itemize}
		\item[\rm (IL1)] $\vdash_{\mathcal L} \alpha \to \alpha$. 
		\item[\rm (IL2)] $\alpha \to \beta, \beta\to \gamma \vdash_{\mathcal L} \alpha \to \gamma$.
		\item[\rm (IL3)] For each connective $f$ in the language of arity $n>0$,\\ $\left\{\begin{array}{c}
		\alpha_1 \to \beta_1, \ldots, \alpha_n \to \beta_n, \\
		\beta_1 \to \alpha_1, \ldots, \beta_n \to \alpha_n
		\end{array}
		\right\} \vdash_{\mathcal L} f(\alpha_1, \ldots, \alpha_n) \to f(\beta_1, \ldots, \beta_n)$.
		\item[\rm (IL4)] $\alpha, \alpha \to \beta \vdash_{\mathcal L} \beta$.
		\item[\rm (IL5)] $\alpha \vdash_{\mathcal L} \beta \to \alpha$.
	\end{itemize}
\end{definition}

\begin{definition} 
(Rasiowa) 
	Let $\mathcal{L}$  be an implicative logic in $\mathbf L$  with $\to$. \\ An {\bf $\mathcal{L}$-algebra} is an algebra $\mathbf A$ in the language $\mathbf L$ that has an element $1$ with the following properties:
	\begin{quote}
		\begin{itemize}
			\item[{\rm (LALG1)}] For all $\Gamma \cup \{\phi\} \subseteq \mbox{\it Fm} $ and all $h \in Hom({\bf Fm}_{\bf L}, \mathbf A)$,\\
			 if $\Gamma \vdash_{\mathcal{L}} \phi$ and $h \Gamma \subseteq \{1\}$ then $h \phi = 1$, 
			\item[{\rm (LALG2)}] For all $a,b \in A$, if $a \to b = 1$ and $b \to a = 1$ then $a=b$.
		\end{itemize}
	\end{quote} 
	
\noindent	The class of $\mathcal{L}$-algebras is denoted by $Alg^*(\mathcal{L})$. 
\end{definition}

\begin{theorem} \label{algebraizabilityOfImplicative}
(The algebraizability theorem of Blok and Pigozzi)\\
Let $\mathcal L$ be an implicative logic and $\mathbb K =Alg^*{\mathcal (L)}$.   $\mbox{For all } \Gamma \cup \{\phi\} \subseteq  \mbox{\it Fm, }$ 
and $\mbox{for all } \Sigma \cup \{\epsilon \approx \delta\} \subseteq  \mbox{\bf Id, }$ the following properties hold:
        \begin{quote}
\begin{enumerate}
\item[{\rm(ALG1)}] 	$\Gamma \vdash_{\mathcal{L}} \phi \mbox{ if and only if }$ 
	$\{\alpha \approx 1: \alpha \in \Gamma\} \models_{Alg^*(\mathcal{L})} \phi \approx 1$. 

\item[{\rm(ALG2)}]    $\Sigma \models_K \epsilon \approx \delta \mbox{ if and only if }$

  \qquad      $ \{\alpha \to \beta, \beta \to \alpha : \alpha \approx \beta \in \Sigma\} \vdash_{\mathcal L} \{ \epsilon \to \delta, \delta \to \epsilon\} $. 
  
\item[{\rm(ALG3)}]   $\phi \vdash\dashv_{\mathcal{L}} \{\phi \to 1, 1 \to \phi\}$.

\item[{\rm(ALG4)}]  $\epsilon \approx \delta \models _{\mathbb K} \{\epsilon \to \delta \approx 1, \delta \to \epsilon \approx 1\}$ and

\qquad  $\{\epsilon \to \delta \approx 1, \delta \to \epsilon \approx 1\} \models_{\mathbb K} \epsilon \approx \delta $.
\end{enumerate}	
\end{quote}
\end{theorem}

\begin{Remark}  
(ALG2) and (ALG3) are derivable from (ALG1) and (ALG4) and conversely.
\end{Remark}

The following theorem, which is a restatement of the preceding theorem,  provides a ``translation algorithm'' to go back and forth between an implicative logic $\mathcal{L}$ and 
$Alg^*(\mathcal{L})$.

\begin{Theorem} (BP89), (Fo16) \label{imp-alg}
Every implicative logic  $\mathcal{L}$ is algebraizable (in the sense of Blok and Pigozzi) with respect to the class $Alg^*(\mathcal{L})$ and the algebraizability is witnessed by the set of defining identities 
$E(x) = \{x \approx x \to x\} $ and the set of equivalence formulas  
$\Delta(x, y) =\{x \to y, y \to x\}$. 
\end{Theorem}

{\bf Axiomatic Extensions of Algebraizable logics:} 
  
A logic $\mathcal{L'}$ is an {\bf axiomatic extension} of $\mathcal{L}$ if $\mathcal{L}$ is obtained by adding new axioms but keeping the rules of inference the same as in $\mathcal{L}$. \\	
Let $Ext(\mathcal{L})$ denote the lattice of axiomatic extensions of a logic $\mathcal{L}$ and 
$\mathbf{L_V(\mathbb V)}$ denote the lattice of subvarieties of a variety $\mathbb V$ of algebras. 

The following theorems due to Blok and Pigozzi \cite{BlPi89} are useful here. 
	
\begin{theorem}\cite{BlPi89} ,   \cite[Theorem 3.33] {Fo16} \label{Ext}
Let $\mathcal{L}$ be an algebraizable logic whose equivalent algebraic semantics $\mathbb V$ is a variety.  Then
$Ext(\mathcal{L})$ is dually isomorphic to $\mathbf{L_V(\mathbb V)}$. \\
\end{theorem} 

\begin{theorem} (Blok and Pigozzi 1989) \label{Ext-algebraizable} 
Let $\mathcal{L}$ be an algebraizable logic and $\mathcal{L'} $  be an axiomatic extension of $\mathbf{\mathit{L}}$.  Then  $\mathcal{L'}$ is also algebraizable with the same transformers $\tau$ and $\rho$ of $\mathcal{L}$.
\end{theorem}

\section{A logic for the variety $\mathbb{RUNO}1$}

In 2022, \cite{CoSa22a} introduced, among others, the logic $\mathcal{DMSH}$ which arose in an attempt to ``logicize'' the variety of De Morgan semi-Heyting algebras. 
This logic will play a crucial role in 
the rest of this paper.  We will now recall this logic from \cite{CoSa22a}.\\

\subsection{{\bf De Morgan Semi-Heyting Logic} }

\

\medskip

\noindent {\bf The logic $\mathcal{DMSH}$}:\\

\smallskip
{\bf \noindent LANGUAGE}: $\langle  \lor,  \wedge,  \to,  \sim, \bot, \top \rangle$\\

Let $\alpha\to_H \beta$ := $\alpha\to(\alpha\land \beta)$ 
and 
 $\alpha \leftrightarrow_H \beta := (\alpha \to _H \beta) \land (\beta \to_H  \alpha)$.\\

{\bf \noindent AXIOMS:}

{\bf \noindent (a) SEMI-HEYTING LOGIC}: 
\begin{itemize}
\item[1] \noindent \rm $\alpha \to_H (\alpha \vee \beta)$, \label{axioma_supremo_izq}  

\item[2] \noindent $\beta \to_H (\alpha \vee \beta)$, \label{axioma_supremo_der}

\item[3] \noindent $(\alpha \to_H  \gamma) \to_H [(\beta \to_H  \gamma) \to_H (\alpha \vee \beta) \to_H \gamma)]$, \label{axioma_supremo_cota_inferior}

\item[4] \noindent $(\alpha \wedge \beta) \to_H {\alpha}$, \label{axioma_infimo_izq}

\item[5] \noindent $(\gamma \to_H  \alpha) \to_H  [(\gamma \to_H  \beta) \to_H  (\gamma \to_H  (\alpha \wedge \beta))]$, \label{axioma_infimo_cota_superior} 

\item[6] \noindent $\top$, \label{axioma_top}

\item[7] \noindent $\bot \to_H \alpha$, \label{axioma_bot}

\item[8] \noindent $[(\alpha \land \beta) \to_H  \gamma] \to_H  [\alpha \to_H  (\beta \to_H  \gamma)]$, \label{axioma_condicRes_InfAImplic}

\item[9] \noindent $[\alpha \to_H  (\beta \to_H  \gamma )] \to_H  [(\alpha \wedge \beta) \to_H  \gamma]$, \label{axioma_condicRes_ImplicAInf}

\item[10] \noindent $(\alpha \to_H  \beta) \to_H  [(\beta \to_H  \alpha) \to_H ((\alpha \to \gamma) \to_H  (\beta \to \gamma))]$, \label{axioma_BuenaDefImplic2} \label{axioma_ImplicADerecha}

\item[11] \noindent $(\alpha \to_H \beta) \to_H  [(\beta \to_H  \alpha) \to_H ((\gamma \to \beta) \to_H  (\gamma \to \alpha))]$. \label{axioma_BuenaDefImplic1} \label{axioma_ImplicAIzquierda}  
\end{itemize}

{\bf (b) De Morgan Axioms}
\begin{itemize}
\item[12]  $\top \leftrightarrow_H  \ \sim \bot$, \label{axiom_bot_neg}  

\item[13]  $\sim (\alpha \wedge \beta) \leftrightarrow_H  \ (\sim \alpha \ \vee \ \sim \beta)$. \label{axiom_DM_law1}  

\item[14] {$\sim(\alpha \vee \beta) \ \leftrightarrow_H  \ (\sim\alpha \ \land \ \sim\beta)$,} \label{axiom_DSDSH_1}
		
\item[15]{$(\sim\sim\alpha) \ \leftrightarrow_H  \ \alpha$.} 
\end{itemize}

\medskip
 {\bf \noindent RULES OF INFERENCE:}\\

\medskip 
	 
	       {\bf (SMP)}  \quad From $\phi$ and $\phi \to_H  {\gamma}$, deduce $\gamma$ \quad ({\rm semi-Modus Ponens}),\\
	       
\medskip 
              
          {\bf (SCP)} \quad From  $\phi \to_H  \gamma$, deduce $\sim \gamma \ \to_H \  \sim \phi$ \quad \rm(semi-Contraposition rule\rm).\\

\medskip 
The following theorem was proved in \cite{CoSa22a}.

\begin{theorem} \label{teorema_DMSH_implicativa} \label{teorema_DHMSH_implicativa}
	The logic $\mathcal{DMSH}$ is an implicative logic with respect to the connective $\to_H$.
\end{theorem}

From Theorem \ref{imp-alg} and Theorem \ref{teorema_DMSH_implicativa},
we have the following corollary which was already mentined in \cite{CoSa22a}.

\begin{Corollary} \label{CorExt} 
 The logic $\mathcal{DMSH}$ is algebraizable and  the variety  
 $\mathbb{DMSH}$ is the equivalent algebraic semantics for $\mathcal{DMSH}$ with the the two transformers $\tau$ and $\rho$ given by  
$\tau(x) =E(x) = \{x \approx x \to_H x\} = \{x \approx 1\}$ and   
$\rho(x, y)=\Delta(x, y) =\{x \to_H y, y \to_H x\} = \{x \leftrightarrow_H y \}$. \\
 \end{Corollary}

\subsection{{\bf The logic $\mathcal{RUNO}1$}}

\

\medskip
The logic $\mathcal{RUNO}1$ will be defined as an axiomatic extension of the logic $\mathcal{DMSH},$ obtained by adding the following additional axioms:

 Define $\neg$ by: $\neg \alpha := \alpha \to 0$. \\   

{\bf (Additional) AXIOMS:}
 
\begin{itemize}

\item[16]  $\sim (0 \to 1) \leftrightarrow_H  (0 \to 1)$ (Unorthodoxy),  

\item[17] $(\alpha \ \wedge \sim\neg\sim\alpha) \vee (\beta \vee \neg \beta) \leftrightarrow_H  (\beta \vee \neg \beta)$  \quad (Regularity),

\item[18] $\neg \sim (\alpha \land \neg \sim \alpha)  \leftrightarrow_H  \alpha \land \neg \sim \alpha$ \quad (Level 1).
\end{itemize}

Since we know the logic $\mathcal{RUNO}1$ is an axiomatic extension of $\mathcal{DMSH}$,
we have the following consequence of Corollary  \ref{CorExt}.

\begin{Corollary} \label{CorExt1} 
 The logic $\mathcal{RUNO}1$ is algebraizable, and  the variety  
 $\mathbb{RUNO}1$ is the equivalent algebraic semantics for $\mathcal{RUNO}1$, with the $E(x)$ and 
 $\Delta(x, y)$ given by 
$E(x) := \{x \approx x \to_H x\} = \{x \approx 1\}$ and   
$\Delta(x, y) :=\{x \to_H y, y \to_H x\} = \{x \leftrightarrow_H y \}$. 
 \end{Corollary}

  We have the following important consequence of Corollary \ref{CorExt1}, Theorem \ref{Ext}  and  Theorem \ref{Ext-algebraizable}.

\begin{Theorem}  
\label{Ext_subvariety_RUNO1} 
\begin{thlist}
\item The lattice $Ext(\mathcal{RUNO}1)$ of axiomatic extensions of the logic $\mathcal{RUNO}1$ is dually isomorphic to the lattice 
$\mathbf{L_V(\mathbb{RUNO}1)}$ of subvarieties of the variety $\mathbb{RUNO}1$.

\item 
Every axiomatic extension of $\mathcal{RUNO}1$ is also algebraizable with its corresponding variety as its equivalent algebraic semantics.
\end{thlist}
\end{Theorem}  
The above theorem justifies the use of the phrase ``the logic corresponding to a subvariety $\mathbb  V$ of $\mathbb{SH}$. 

\smallskip

Let $\mathbf{Mod}(\mathcal E) := \{\mathbf{A} \in \mathbb{SH}: \mathbf{A} \models \delta \approx 1, \mbox{ for every } \delta \in \mathcal E\} $.  

\smallskip
The following theorem is an immediate consequence of Corollary \ref{CorExt1} and plays an important role in the next section.

 \begin{Theorem} \label{completeness_SH_extension} \label{teo_040417_01}
 	Let $\mathcal{E}$ be an axiomatic extension of the logic $\mathcal{SH}$.  Then          
	
   {\rm (a)}  $\mathcal{E}$ is also algebraizable with the same equivalence formulas and defining equations as those of the logic $\mathcal{SH}$.

   {\rm (b)}  $\mathbf{Mod}(\mathcal E)$ is an equivalent algebraic semantics of $\mathcal E$.   
 \end{Theorem}

\medskip
\section{Unorthodox Logics: Axiomatic Extensions of the logic $\mathcal{RUNO}1$}
\smallskip

Let us call the axiomatic extensions of the logic $\mathcal{RUNO}1$ as ``unorthodox logics''.
In this section, we provide axiomatizations (i.e., bases) for these unorthodox logics.  
The presentation of the bases is done according to the levels mentioned in Corollary \ref{CorLat} of the lattice of subvarieties of $\mathbb {RUNO}1$.

In the following theorem, we let $\alpha :=\bot \to \top$, $\beta :=\bot \to \alpha$ and $\gamma := \bot \to \beta$ 
\begin{Theorem}
\begin{thlist}
\item[1] {\bf BASES FOR LOGICS ASSOCIATED WITH SINGLE ALGEBRAS}  
\begin{thlist}
\item[i] {\bf A BASE FOR the logic $\mathcal{A}1$ of $\mathbb{V}(\mathbf{A1})$: }

 \qquad \quad (a)  $\alpha \to \top$,   
 
 \qquad \quad (b)  $ \beta$.

\item[ii] {\bf A BASE FOR the logic $\mathcal{A}2$ of $\mathbb{V}(\mathbf{A2})$: }

\qquad \quad (a) $\beta$  

\qquad \quad (b) $\alpha \to \top \leftrightarrow_H  \alpha.$

\qquad \quad (c) $\neg \alpha \leftrightarrow_H  \bot$\\

\item[iii] {\bf A BASE FOR the logic $\mathcal{A}3$ of $\mathbb{V}(\mathbf{A3})$: }

 \qquad \quad (a) $\phi \to (\phi \to (\phi \to \psi)) \leftrightarrow_H  \phi \to (\phi \to \psi)$.\\

\item[iv] {\bf A BASE FOR the logic $\mathcal{A}4$ of $\mathbb{V}(\mathbf{A4})$: }

 \qquad \quad $\sim (\alpha \to \top) \  \leftrightarrow_H  \ \beta$.\\

\item[v] {\bf A BASE FOR the logic $\mathcal{A}5$ of $\mathbb{V}(\mathbf{A5})$: }

\qquad \quad $\neg \alpha \to \top$.  
\end{thlist}

\medskip

\item[2] {\bf BASES FOR LOGICS OF PAIRS OF ALGEBRAS}      

\smallskip
\begin{thlist}
\item[i] {\bf A BASE FOR the logic of $\mathbb{V}(\{\mathbf A1, \mathbf A2\})$ :}

\qquad \quad (a) $\neg \alpha \leftrightarrow_H  \bot$,

 \qquad \quad (b) $\beta$.\\ 

\medskip
\item[ii] {\bf A BASE FOR the logic of $\mathbb{V}(\{\mathbf A1, \mathbf A3\})$ :}

\qquad \quad (a) $\neg \alpha \leftrightarrow_H  \bot$,

\item[iii] {\bf A BASE FOR the logic of $\mathbb{V}(\{\mathbf A1, \mathbf A4\})$ :} 

\qquad \quad (a) $\neg\alpha \leftrightarrow_H  \bot$,

 \qquad \quad (b) $\beta \leftrightarrow_H  \alpha \to \top$,

\item[iv] {\bf A BASE FOR the logic of $\mathbb{V}(\{\mathbf A1, \mathbf A5\})$ :}

\qquad \quad (a) $\beta \to  \neg\neg\alpha\  \leftrightarrow_H  \  \neg\neg\alpha,$

 \qquad \quad (b) $\beta \ \geq \ \alpha \to \top$,
  
 \qquad \quad (c) $\neg\alpha \to (\alpha \to \top) \leftrightarrow_H \alpha.$

\item[v] {\bf A BASE FOR the logic of $\mathbb{V}(\{\mathbf A2, \mathbf A3\})$:}

\qquad \quad (a) $\neg\alpha \leftrightarrow_H  \bot$,

 \qquad \quad (b) $\beta \to  \neg\neg\alpha\  \leftrightarrow_H  \  \neg\neg\alpha,$
 
\qquad \quad (c) $(\alpha \to \beta) \to \alpha \leftrightarrow_H  \beta.$\\

\item[vi] {\bf A BASE FOR the logic of $\mathbb{V}(\{\mathbf A2, \mathbf A4\})$ :}

\qquad \quad (a) $\neg\alpha \leftrightarrow_H  \bot$,

 \qquad \quad (b) $\beta \ \geq \ \alpha \to \top$,
  
\qquad \quad (c) $(\alpha \to \beta) \to \alpha \leftrightarrow_H  \beta.$\\  

\item[vii] {\bf A BASE FOR the logic of $\mathbb{V}(\{\mathbf A2, \mathbf A5\}$ : }

  \qquad  \quad (a) $\alpha \to \beta  \leftrightarrow_H  \alpha,$ 
  
   \qquad \quad (b) $\beta \ \geq \ \alpha \to \top$,
   
  \qquad \quad (c) $(\alpha \to \beta) \to \alpha \leftrightarrow_H  \beta.$ \\

\item[viii] {\bf A BASE FOR the logic of $\mathbb{V}(\{\mathbf A3, \mathbf A4\}$:}

\qquad \quad (a) $\neg\alpha \leftrightarrow_H  \bot$,

 \qquad \quad (b) $\beta \leftrightarrow_H  \alpha,$
 
\qquad \quad (c) $(\alpha \to \beta) \to \alpha \leftrightarrow_H  \beta.$\\  

\item[ix] {\bf A BASE FOR the logic of $\mathbb{V}(\{\mathbf A3, \mathbf A5\}$: }

 \qquad \quad (a) $\beta \to \neg\neg \alpha  \leftrightarrow_H   \neg\neg \alpha,$
 
 \qquad \quad (b) $\neg \alpha \to (\alpha \to \top) \leftrightarrow_H  \alpha$,
 
\qquad \quad (c) $(\alpha \to \beta) \to \alpha \leftrightarrow_H  \beta.$\\  

\item[x] {\bf A BASE FOR the logic of $\mathbb{V}(\{\mathbf A4, \mathbf A5\}$:} 

  \qquad \quad (a) $\beta \ \geq \ \alpha \to \top$,
  
   \qquad \quad (b) $\neg \alpha \to (\alpha \to \top) \leftrightarrow_H  \alpha$,
   
  \qquad \quad (c) $(\alpha \to \beta) \to \alpha \leftrightarrow_H  \beta.$\\   
\end{thlist}  
  
\item[3] {\bf BASES FOR LOGICS OF TRIPLES OF ALGEBRAS} 

\begin{thlist}
\item[i] {\bf A BASE FOR the logic of $\mathbb{V}(\{\mathbf A1, \mathbf A2, \mathbf A3\}$:}

\qquad \quad (a) $\neg\alpha \leftrightarrow_H  \bot$,

 \qquad \quad (b) $\beta \to \neg\neg \alpha  \leftrightarrow_H   \neg\neg \alpha.$

\item[ii] {\bf A BASE FOR the logic of $\mathbb{V}(\{\mathbf A1, \mathbf A2, \mathbf A4\}$:}

\qquad \quad (a) $\neg\alpha \leftrightarrow_H  \bot$,

\qquad \quad (b) $\beta \ \geq \ \alpha \to \top$.

\item[iii] {\bf A BASE FOR the logic of $\mathbb{V}(\{\mathbf A1, \mathbf A2, \mathbf A5\}$:}

\qquad \quad (a) $\beta \to \neg\neg \alpha  \leftrightarrow_H   \neg\neg \alpha,$

\qquad \quad (b) $\beta \ \geq \ \alpha \to \top$.

\item[iv] {\bf A BASE FOR the logic of $\mathbb{V}(\{\mathbf A1, \mathbf A3, \mathbf A4 \}$:}

\qquad \quad (a) $\neg\alpha \leftrightarrow_H  \bot$;

 \qquad \quad (b) $\gamma \leftrightarrow_H  \alpha;$

 \qquad \quad (c) $\neg\alpha \to (\alpha \to \top) \leftrightarrow_H  \alpha.$\\

\item[v] {\bf A BASE FOR the logic of $\mathbb{V}(\{\mathbf A1, \mathbf A3, \mathbf A5\}$:}

 \qquad \quad (a) $\neg\alpha \to (\alpha \to \top) \leftrightarrow_H  \alpha,$
 
 \qquad \quad (b) $\beta \to \neg\neg \alpha  \leftrightarrow_H   \neg\neg \alpha.$

\item[vi] {\bf A BASE FOR the logic of $\mathbb{V}(\{\mathbf A1, \mathbf A4, \mathbf A5\}$:}

 \qquad \quad (a) $\beta \ \geq \ \alpha \to \top$,
 
 \qquad \quad (b) $\neg\alpha \to (\alpha \to \top) \leftrightarrow_H  \alpha.$\\

\item[vii] {\bf A BASE FOR the logic of $\mathbb{V}(\{\mathbf A2, \mathbf A3, \mathbf A4\}$:}

\qquad \quad $\neg\alpha \leftrightarrow_H  \bot$,

\qquad \quad $(\alpha \to \beta) \to \alpha \leftrightarrow_H  \beta.$\\  

\item[viii] {\bf A BASE FOR the logic of $\mathbb{V}(\{\mathbf A2, \mathbf A3, \mathbf A5\}$:}

 \qquad \quad (a) $\beta \to \neg\neg \alpha  \leftrightarrow_H   \neg\neg \alpha,$
 
\qquad \quad (b) $(\alpha \to \beta) \to \alpha \leftrightarrow_H  \beta.$\\

\item[ix] {\bf A BASE FOR the logic of $\mathbb{V}(\{\mathbf A2, \mathbf A4, \mathbf A5 \}$:}

  \qquad \quad $\beta \ \geq \ \alpha \to \top$,
  
\qquad \quad $(\alpha \to \beta) \to \alpha \leftrightarrow_H  \beta.$\\ 

\item[x] {\bf A BASE FOR the logic of $\mathbb{V}(\{\mathbf A3, \mathbf A4, \mathbf A5\}$:}

 \qquad \quad (a) $\neg\alpha \to (\alpha \to \top) \leftrightarrow_H  \alpha,$
 
\qquad \quad (b) $(\alpha \to \beta) \to \alpha \leftrightarrow_H  \beta.$\\  

\end{thlist}

\item[4] {\bf BASES FOR LOGICS OF QUADRUPLES OF ALGEBRAS}:
\begin{thlist}
\item[i] {\bf A BASE FOR the logic of $\mathbb{V}(\{\mathbf A1, \mathbf A2, \mathbf A3, \mathbf A4\}$:}

\qquad \quad $\neg\alpha \leftrightarrow_H  \bot$,\\

\item[ii] {\bf A BASE FOR the logic of $\mathbb{V}(\{\mathbf A1, \mathbf A2, \mathbf A3, \mathbf A5\}$:}

 \qquad \quad $\beta \to \neg\neg \alpha  \leftrightarrow_H   \neg\neg \alpha,$\\

\item[iii] {\bf A BASE FOR the logic of $\mathbb{V}(\{\mathbf A1, \mathbf A2, \mathbf A4, \mathbf A5\}$:}

  \qquad \quad $\beta \ \geq \ \alpha \to \top$. \\

\item[iv] {\bf A BASE FOR the logic of $\mathbb{V}(\{\mathbf A1, \mathbf A3, \mathbf A4, \mathbf A5\}$: }

 \qquad \quad $\neg\alpha \to (\alpha \to \top) \leftrightarrow_H  \alpha.$\\

\item[v] {\bf A BASE FOR the logic of $\mathbb{V}(\{\mathbf A2, \mathbf A3, \mathbf A4, \mathbf A5\}$:}

\qquad \quad $(\alpha \to \beta) \to \alpha \leftrightarrow_H  \beta.$ 
\end{thlist}
\end{thlist}
\end{Theorem}

Observe that all the logics mentioned in the preceding theorem are implicative since they are axiomatic extensions of $\mathcal{RUNO}1$.  Hence the following corollary is immediate.

\begin{Corollary}
All the logics described in the above theorem are algebraizable.
\end{Corollary} 

\begin{Corollary}
All the logics corresponing to the subvarieties of $\mathbb{RUNO}1$ are decidable.
\end{Corollary}

Using a slight generalization (see \cite{CoSa23b}) of Maximova's algebraic characterization of Disjunction Proprty, we can prove the following theorem.

\begin{Theorem}
Every consistent axiomatic extension of the logic  $\mathbb{RUNO}1$ does not have the Disjunction Property.
\end{Theorem}

\medskip
\section{CONCLUDING REMARKS}

In the preceding sections we saw:
\begin{thlist}
\item[a]  $\mathbb{V}(\mathbf{A1}, \mathbf{A2}, \mathbf{A3}, \mathbf{A4}, \mathbf{A5})
= \mathbb{RUNO}1$.

\item[b]  $\mathbb{RUNO}1$ is a discriminator variety and its first-order theory is decidable.

\item[c]  The algebras $\mathbf{Ai}$, ${\bf i} =1,2, \cdots, 5$ are primal. 

\item[d]  Each of the varieties $\mathbb{V}(\mathbf{Ai})$, ${\bf i} =1,2, \cdots, 5$ has the SAP and ES. 

\item[e] The lattice of subvarieties of $\mathbb{RUNO}1$ is a Boolean lattice consisting of 32 subvarieties of $\mathbb{RUNO}1$.

\item[f] Equational axioms for all subvarieties of $\mathbb{RUNO}1$.

\item[g] The logic $\mathcal{RUNO}1$ for the variety $\mathbb{RUNO}1$.

\item[h]  Axiomatizations for 
all axiomatic extensions of $\mathcal{RUNO}1$).
\end{thlist}

The items (b), (c) and (d) above imply that the algebras $\mathbf{Ai}$, ${\bf i} =1,2, \cdots, 5$ 
share some strong properties with the Boolean algebra $\mathbf 2$, from which we can conclude that each of the algebras $\mathbf{Ai}$, ${\bf i} =1,2, \cdots, 5$ (and their logic) could be a formidable competitor to the Boolean algebra $\mathbf{2}$ in computing.
Computer chips based on three-valued logic of any (or all in parallel) of  $\mathbf{Ai}$, ${\bf i} =1,2, \cdots, 4$, or based on four-valued logic of $\mathbf{A5}$, could be more useful in future, since they would have more expressive power than the binary chips.  

\medskip

\noindent {\bf OPEN PROBLEMS}:
Here are some open problems for further research.\\

{\bf PROBLEM 1}:  Find a duality (Priestley-type or otherwise) for each of the subvarieties of  $\mathbb{RUNO}1$. \\ 

\smallskip
{\bf PROBLEM 2}: Let $\mathbb{RUNO}2$ denote the variety of regular unorthodox algebras of level 2 (that is, satisfying the identity:  $(x \land x'^*)'{^*}'^* \approx  (x \land x'^*)'{^*}).$
Characterize the simple algebras in $\mathbb{RUNO}2$.

\smallskip
{\bf PROBLEM 3}: 
Characterize the simple algebras in the variety $\mathbb{UNO}1$.

\smallskip
{\bf PROBLEM 4}: Call an algebra $\mathbf A \in \mathbb{DMSH}1$ {\bf contra-Boolean} if  $\mathbf A \models 0 \to 1 \approx 0$.  Let $\mathbb{RCB}1$ denote the subvariety of $\mathbb{DMSH}1$ consisting of regular contra-Boolean algebras of level 1.
Characterize the simple algebras in the variety $\mathbb{RCB}1$. \\

{\bf A Historical Observation:}

The fact that the first appearance of non classical logics only occurred around the end of the 19th century and in the beginning of the 20th century--more than 2000 years after Aristotle--seems to suggest that, perhaps, Question B was not considered until recently.  In view of this observation, the following problem (or Question B, for that matter) might be of philosophical and/or historical interest:\\

{\bf PROBLEM:}  Why did the discovery of nonclassical logics not happen for more than 2000 years after Aristotle?\\

\end{document}